\documentclass[12pt, oneside,reqno]{amsart}
\usepackage{Everything}

\title{Characterization of $\LL^1_\kappa$}
\date{March 2023}
\author{Siiri Kivim\"{a}ki}
\address[S.  Kivim\"{a}ki]{
Department of Mathematics and Statistics, University of Helsinki, Finland}
\email{siiri.kivimaki@helsinki.fi}

\author[B. Veli\v{c}kovi\'c]{Boban Veli\v{c}kovi\'c}
\address[B. Veli\v{c}kovi\'c]{
Institut de Math\'ematiques de Jussieu - Paris Rive Gauche (IMJ-PRG)\\
Universit\'e Paris Cit\'e\\
B\^atiment Sophie Germain\\
8 Place Aur\'elie Nemours \\ 75013 Paris, France}
\email{boban@math.univ-paris-diderot.fr}
\urladdr{http://www.logique.jussieu.fr/~boban/}

\begin{document}

\maketitle

\begin{abstract}
    The logic $\LL^1_\kappa$ was introduced by Shelah in \cite{shelah2012nice}. In \cite{shelah2021isomorphic}, he proved that for a strongly compact cardinal $\kappa$, it admits the following algebraic characterization: two structures are $\LL^1_\kappa$-equivalent if and only if they have isomorphic iterated ultrapowers via $\kappa$-complete ultrafilters. We give  presentation of the logic $\LL^1_\kappa$ and a simplified and slightly modified proof of this result.
\end{abstract}

\section{The logic $\LL^1_{\kappa}$}

The logic $\LL^1_{\kappa}$ is defined through a variation of an Ehrenfeucht-Fraïssé game. The \textit{states} of this game will be triples $(\alpha,f,\pi)$, where $\alpha$ is an ordinal, $\pi$ is a partial isomorphism, and $f$ is a partition function which partitions some subset of the field of $\pi$ into countably many pieces.

\begin{de*}[The game $\G^\beta_\theta$] Let $\AA$ and $\BB$ be structures of same signature, let $\beta$ be an ordinal and let $\theta$ be a cardinal. The game \[
\G^\beta_\theta(\AA,\BB)
\]
is played as follows.

\textbf{Starting state:} The starting state is $(\beta,\emptyset,\emptyset)$.

\textbf{Further states:} Assume that the game is at state $(\alpha, f,\pi)$. 
    \begin{itemize}
        \item The player $\1$ chooses some ordinal $\alpha'<\alpha$ and some set $X\in\AA^{\leq\theta}\cup\BB^{\leq\theta}$.
        \item The player $\2$ chooses a partial partition function $f':\AA\cup\BB\to\omega$ such that ${\dom(f),X\subseteq \dom(f')}$
        and such that for all $a\in \dom(f)$, \[
        f'(a):=f(a)\dot-1.
        \]
        Then she chooses a partial isomorphism $\pi'\supseteq\pi$ such that \[{f'^{-1}\{0\}\subseteq \fld(\pi')}.\] 
        The next state is $(\alpha',f',\pi')$.
    \end{itemize}

The player to first break the rules loses.
\end{de*}

Let $\equiv^\beta_\theta$ be the transitive closure of the relation 
\[
\text{The player }\2\text{ has a winning strategy in the game }\G^\beta_\theta(\AA,\BB).
\]

A \textit{logic} is a class function associating to each signature $\tau$ a collection of sentences and a satisfaction relation, satisfying certain regularity properties, see \cite{Barwise1985-BARML-8}.

\begin{de*}[The logic $\LL^1_\kappa$] Let $\tau$ be a signature.
\begin{enumerate}
    \item A \textit{$\tau$-sentence} in $\LL^1_\kappa$ is a class of $\tau_0$-structures which is closed under the relation $\equiv^\beta_\theta$, for some $\tau_0\in[\tau]^{<\kappa}$ and some $\beta,\theta<\kappa$.
    \item The satisfaction relation of $\LL^1_\kappa$ is defined as \[
\MM\models\phi\quad:\iff\quad\MM\rest\tau_0\in\phi,
\]
where $\tau_0$ is the signature such that $\phi$ consists of $\tau_0$-structures.
\end{enumerate}

\end{de*}

\begin{fact*} For cardinals of the form $\kappa=\beth_\kappa$, the logic $\LL^1_\kappa$ is a regular logic strictly between the logics $\LL_{\kappa\omega}$ and $\LL_{\kappa\kappa}$.
\end{fact*}
\begin{proof}
    See \cite{shelah2012nice}.
\end{proof}

Notice that for any $\tau$-structures $\AA$ and $\BB$, \[
\AA\equiv_{\LL^1_\kappa}\BB\quad\iff\quad\forall \tau_0\in[\tau]^{<\kappa}\quad\forall\beta,\theta<\kappa\quad \AA\rest\tau_0\equiv^\beta_\theta\BB\rest\tau_0.
\]

\begin{prop}[The Union Lemma for $\LL^1_\kappa$]\label{union} Assume that $\kappa=\beth_\kappa$. Assume that $\bar{\AA}=\{\AA_n\}_{n\in \omega}$ is an $\LL_{\kappa\kappa}$-elementary chain of structures. Then, for each $n$,
\[
\AA_n\equiv_{\LL^1_{\kappa}}\bigcup\bar{\AA}.
\]
\end{prop}
\begin{proof}
    See \cite{shelah2012nice}.
\end{proof}

\section{Ultrapowers and $\LL^1_\kappa$-theories}

If $\UU$ is an ultrafilter on a set $I$ and $\AA$ is a structure, the ultrapower $\AA^{I}/\UU$ will be denoted by $\AA^\UU$. If $\UUU=(\UU_n)_n$ is a sequence of ultrafilters on some sets and $\AA$ is a structure, the iterated ultrapower of $\AA$ along the ultrafilters $(\UU_n)_n$ will be denoted by $\AA^{\UUU}$. In other words, $\AA^{\UUU}$ is the direct limit of the system 
\[
(\AA_n,j_{m,n})_{m<n<\omega},
\]
where \begin{align*}
    &\AA_0:=\AA\\
    &\AA_{n+1}:=\AA_n^{\UU_n},
\end{align*}
and the maps $j_{m,n}:\AA_m\to\AA_n$ are compositions of the ultrapower embeddings. In case the ultrafilters are $\kappa$-complete, we have:

\begin{thm*}[\L os] If $\UU$ is a $\kappa$-complete ultrafilter on a set $I$ and $\AA$ is a structure, then the ultrapower embedding \[
\AA\to\AA^\UU,\quad a\mapsto[(a)_{i\in I}]_\UU
\]
is $\LL_{\kappa\kappa}$-elementary.
    
\end{thm*}

By the \L os Theorem, thus, if the ultrafilters $\UU_n$ are $\kappa$-complete, then the maps $j_{m,n}:\AA_m\to\AA_n$ are $\LL_{\kappa\kappa}$-elementary. 

The direct limit $\AA^{\UUU}$ comes together with embeddings \[
j_{n,\omega}:\AA_n\to\AA^{\UUU},
\]
which are first-order elementary but not more in general. In particular, the limit embeddings might fail to be $\LL_{\kappa\kappa}$-elementary, even if the ultrafilters were $\kappa$-complete. In this case, they might even fail to be $\LL^1_\kappa$-elementary, but by the Proposition \ref{union}, they still preserve the $\LL^1_\kappa$-theory.

For instance, any ultrapower of a well-founded model by a $\kappa$-complete ultrafilter is again well-founded, since well-foundedness is expressible in the logic $\LL_{\omega_1\omega_1}$, and thus preserved under $\LL_{\kappa\kappa}$-elementary embeddings (in case $\kappa$ is uncountable). However, it is easy to produce an ill-founded model from a well-founded one by iterating the ultrapower construction $\omega$ many times, as will be done in the proof of characterization of $\LL^1_\kappa$.

\subsection*{Strongly compact cardinals} For cardinals $\lambda\geq\kappa$, denote \[
\Pow_\kappa(\lambda):=\{x\subseteq\lambda:|x|<\kappa\}.
\]
An ultrafilter $\UU$ on $\Pow_\kappa(\lambda)$ is \textit{fine} if it is $\kappa$-complete and for each $x\in\Pow_\kappa(\lambda)$, it contains the cone
\[
C_x:=\{y\in\Pow_\kappa(\lambda):x\subseteq y\}.
\]

A cardinal $\kappa$ is \textit{$\lambda$-compact} if there exists a fine ultrafilter on $\Pow_\kappa(\lambda)$. A cardinal $\kappa$ is \textit{strongly compact} if it is $\lambda$-compact for every $\lambda\geq\kappa$. The $\lambda$-compact cardinals have the following covering property:

\begin{lem}\label{lemma} Assume that $\kappa$ is a $\lambda$-compact cardinal and $\UU$ is a fine ultrafilter on $\Pow_\kappa(\lambda)$. Assume that $(H,\in)$ is a transitive model of $\ZFC^-$ closed under $<\kappa$-sequences such that $\kappa,\lambda\in H$. For any set $Y\subseteq H^\UU$ of size at most $\lambda$, there is a set $X\in H^\UU$ such that 
\[
Y\subseteq X\quad\text{and}\quad H^\UU\models|X|< j(\kappa),
\]
where $j:H\to H^\UU$ is the ultrapower embedding.
\end{lem}
\begin{proof}
    Let $Y\subseteq H^\UU$ be a set of size at most $\lambda$. We find a set $X\in H^\UU$ which covers $Y$ and for which \[
    H^\UU\models|X|<j(\kappa).
    \]
    Say $Y=\{[f_i]_\UU:i<\lambda\}$. Define the function $F:\Pow_\kappa(\lambda)\to H$, \[
    F(x)=\{f_i(x):i\in x\}.
    \]
    As $H$ is closed under $<\kappa$-sequences, this function $F$ has indeed its range inside $H$, thus $[F]_\UU\in H^\UU$.  Let $X:=[F]_\UU$. 
    
    By fineness we have $Y\subseteq X$: for each $i<\lambda$, 
    \[
    {C_{\{i\}}\subseteq\{x:f_i(x)\in F(x)\}\in\UU}.
    \] 
    Also ${H^\UU\models|X|< e(\kappa)}$: simply because 
    \[
    \{x:|F(x)|<\kappa\}=\Pow_\kappa(\lambda)\in\UU.
    \]
    
\end{proof}

\section{Proof of the characterization}

We now give a proof of the following theorem.

\begin{thm*}[Shelah, Theorem 1.5 in \cite{shelah2021isomorphic}] Assume that $\kappa$ is a strongly compact cardinal. The following are equivalent:
\begin{enumerate}
    \item $\AA\equiv_{\LL^1_\kappa}\BB$.
    \item There is a sequence $\bar{\UU}=(\UU_n)_{n<\omega}$ of $\kappa$-complete ultrafilters such that \[
    \AA^{\bar{\UU}}\cong\BB^{\bar{\UU}}.
    \]
\end{enumerate}
\end{thm*}

\begin{proof}\spa
\begin{enumerate}
    
    \item[\textit{(1)$\To$(2)}:] Assume that $\AA\equiv_{\LL^1_\kappa}\BB$. For simplicity, assume that the signature $\tau$ of the models $\AA$ and $\BB$ is relational and of size $<\kappa$, and the domains of $\AA$ and $\BB$ are disjoint. For simplicity again, assume that for all $\beta,\theta<\kappa$, the player $\2$ has a winning strategy in the game \[
    \G^\beta_\theta(\AA,\BB).
    \]

    We will build a countable sequence of ultrafilters $\UUU$ such that the iterated ultrapowers $\AA^{\UUU}$ and $\BB^{\UUU}$ are isomorphic.

    Let $\mu$ be a regular cardinal large enough such that the models $\AA$ and $\BB$, $\kappa$, and all the winning strategies are in $H(\mu)$. For all $\beta,\theta<\kappa$, fix some winning strategy $\sigma_{\beta,\theta}$ for the player $\2$ in the game $\G^\beta_\theta(\AA,\BB)$. Choose new unary predicate symbols $A$ and $B$ and a new binary function symbol $\sigma$. Define the structure
    \[
    \HH:=(H(\mu),\in,A^\HH,B^\HH,\sigma^\HH,R^\HH)_{R\in\tau}
    \]
    where
    \begin{itemize}
        \item $A^\HH=\dom(\AA)$
        \item $B^\HH=\dom(\BB)$
        \item $\sigma^\HH(\beta,\theta)=\begin{cases}
            \sigma_{\beta,\theta},\quad&\text{if }\beta,\theta\in\kappa\\
            \emptyset &\text{otherwise.}
        \end{cases}$
        \item For each symbol $R\in\tau$, $R^\HH=R^{\AA}\cup R^\BB$.
    \end{itemize}
    
We will now build structures $(\HH_n)_n$, $(\AA_n)_n$, $(\BB_n)_n$, ultrafilters $(\UU_n)_n$ and sets $(X_n)_n$, by recursion on $\omega$.

    \begin{enumerate}
        \item[\textbf{Step $0$:}] Let $\HH_0:=\HH$, $\AA_0:=\AA$ and $\BB_0:=\BB$.
        
        \item[\textbf{Step $n+1$:}] 
        Assume that $\HH_m$, $\AA_m$ and $\BB_m$ have been defined for all $m\leq n$. For each $m\leq n$, denote \[
        \lambda_m:=|\AA_m|+|\BB_m|+\kappa.
        \]
        Furthermore, assume that for all $m<n$, we have defined (using the fact that $\kappa$ is strongly compact) \begin{itemize}
            \item A fine ultrafilter $\UU_m$ on the set $\Pow_\kappa(\lambda_m)$.
            \item Its corresponding ultrapower embedding \[
            e_m:\HH_m\to\HH_m^{\UU_m}=:\HH_{m+1}.
            \]
            \item A set $X_m\in\HH_{m+1}$ such that the pointwise images $e_m[\AA_m]$ and $e_m[\BB_m]$ are subsets of $X_m$ and \[
            \HH_{m+1}\models|X_m|<e_m(\kappa),
            \] using the covering property of compact cardinals as in Lemma \ref{lemma}.
        \end{itemize}
We now define the ultrafilter $\UU_n$, the model $\HH_{n+1}$, an embedding $e_n$, the set $ X_n $, and the models $ \AA_{n+1}$ and $ \BB_{n+1}$.
\begin{itemize}
\item Let $\UU_{n}$ be any fine ultrafilter on $\Pow_\kappa(\lambda_n)$. This is possible because $\kappa$ is strongly compact.
\item Let $\HH_{n+1}:=\HH_n^{\UU_n}.$
\item Let $e_n:\HH_n\to\HH_{n+1}$
be the ultrapower embedding. Notice that this embedding is $\LL_{\kappa\kappa}$-elementary and its critical point is $\kappa$.
\item Let $X_n\in\HH_{n+1}$ be a set such that 
\[
e_n[\AA_n],e_n[\BB_n]\subseteq X_n\quad\text{and}\quad \HH_{n+1}\models |X_n|<e_n(\kappa).
\]
This is possible by the covering properties of $\lambda_n$-compact cardinals, by Lemma \ref{lemma}.
\item Finally, let
\begin{align*}
    & \AA_{n+1}:=\AA_n^{\UU_n}\\
    &\BB_{n+1}:=\BB_n^{\UU_n}.
\end{align*}

\end{itemize}

\end{enumerate}

We have the directed system
\[
\left(\HH_n,e_{m,n}\right)_{m<n<\omega},
\]
where each $e_{m,n}:\HH_m\to\HH_n$ is an $\LL_{\kappa\kappa}$-elementary embedding, obtained by composing the ultrapower embeddings. Let $\HH^{\UUU}$ be the direct limit of this system.

The restricted maps
\begin{align*}
    &e^\AA_{m,n}:=e_{m,n}\rest\AA_m:\AA_m\to\AA_n\\
    &e_{m,n}^\BB:=e_{m,n}\rest\BB_m:\BB_m\to\BB_n,
\end{align*}
are also $\LL_{\kappa\kappa}$-elementary.
We get the directed systems 
\[
\left(\AA_n,e^\AA_{m,n}\right)_{m<n<\omega}\quad\text{and}\quad\left(\BB_n,e^\BB_{m,n}\right)_{m<n<\omega},
\]
and we can take the direct limits of these systems, denote them by $\AA^{\UUU}$ and $\BB^{\UUU}$, respectively.

We have the first-order elementary limit embeddings:
    \begin{align*}
    &e_{n,\omega}:\HH_n\to\HH^{\UUU}\\
    &e^\AA_{n,\omega}:\AA_n\to\AA^{\UUU}\\
    &e^\BB_{n,\omega}:\BB_n\to\BB^{\UUU}.
\end{align*}
\end{enumerate}

\begin{claim*}
    The models $\AA^{\UUU}$ and $\BB^{\UUU}$ are isomorphic.
\end{claim*}
\begin{proof}[Proof of Claim]

Notice first that for each $n$, the $n$th iterates $\AA_n$ and $\BB_n$ are isomorphic to the structures $A^{\HH_n}$ and $B^{\HH_n}$, 
respectively. Thus also \[
\AA^{\UUU}\cong A^{\HH^{\UUU}}\quad\text{and}\quad \BB^{\UUU}\cong B^{\HH^{\UUU}}.
\]
It is thus enough to show that $A^{\HH^{\UUU}}$ and $B^{\HH^{\UUU}}$ are isomorphic.

By the first-order elementarity of the map $e_{0,\omega}$,
\begin{align*}
    \HH^{\UUU}\models\quad \q\forall\beta,\theta<e_{0,\omega}(\kappa)\quad&\sigma^{\HH^{\UUU}}(\beta,\theta)\textit{ is a winning strategy for the player }\2\textit{ in }\\
        &\textit{the game }\G^\beta_\theta(A^{\HH^{\UUU}},B^{\HH^{\UUU}})\q.
\end{align*}

We now fix some parameters $\beta$ and $\theta$ below $e_{0,\omega}(\kappa)$ in order to consider the game $\G^\beta_\theta(A^{\HH^{\UUU}},B^{\HH^{\UUU}})$, computed in $\HH^{\UUU}$.

Let $\beta:=e_{1,\omega}(\kappa)$ and for each $n$, denote \[
\beta_n:=e_{n+2,\omega}(\kappa).
\]
The sequence $(\beta_n)_n$ is a descending sequence of ordinals of $\HH^{\UUU}$ below $\beta$. 

For each $n$, denote
\[
\bar{X}_n:=e_{n+1,\omega}(X_n).
\]
By construction, the sets $(\bar{X}_n)_n$ cover the domains of the models $\AA^{\UUU}$ and $\BB^{\UUU}$.
Let
\[
\theta:=\max\{|\bar{X}_n|^{\HH^{\UUU}},\beta\}.
\] 
Both $\beta$ and $\theta$ are below $e_{0,\omega}(\kappa)$, and each $\bar{X}_n$ has size $\leq\theta$ in $\HH^{\UUU}$.

Then we describe a play of the player $\1$ in the game $\G^\beta_\theta(\AA^{\UUU},\BB^{\UUU})$:
\begin{itemize} 
    \item At the $(2n+1)$th step, he plays the ordinal $\beta_{2n+1}$ and the set $\AA^{\UUU}\cap \bar{X}_{2n+1}$.
    \item At the $(2n+2)$th step, he plays the ordinal $\beta_{2n+2}$ and the set $\BB^{\UUU}\cap \bar{X}_{2n+2}$.
\end{itemize}
Every finite initial segment of this play is as an element in the model $\HH^{\UUU}$.
Hence, the player $\2$ must be able to win against this play; otherwise, there would be some finite play of the player $\1$ which the player $\2$ loses and this would contradict the fact that in the model $\HH^{\UUU}$, the player $\2$ has a winning strategy in the game $\G^\beta_\theta(A^{\HH^{\UUU}},B^{\HH^{\UUU}})$.

She can thus win, and eventually, after $\omega$ many steps, she will have produced a chain of partial isomorphisms $(\pi_n)_n$ such that 
\[
\bigcup_n\pi_n:A^{\HH^{\UUU}}\cong B^{\HH^{\UUU}}.
\]

This ends the proof of the Claim.

\end{proof}
\item[\textit{(2)$\To$(1)}:] Assume that $\bar{\UU}=(\UU_n)_n$ are $\kappa$-complete ultrafilters, each $\UU_n$ on some set $I_n$, and 
    $\AA^{\UUU}\cong\BB^{\UUU}$.
    We show that $\AA\equiv_{\LL^1_\kappa}\BB$.
    
    Denote 
    \[
\begin{cases} \AA_0:=\AA\\
\AA_{n+1}:=\AA_n^{\UU_n}
\end{cases}
\]
and
\[
\begin{cases} \BB_0:=\BB\\
\BB_{n+1}:=\BB_n^{\UU_n}.
\end{cases}
\]
    Without loss of generality we may identify each $\AA_n$ with its image under the embedding into the direct limit and get that for each $n$,
    \[
    \AA_n\elem_{\LL_{\kappa\kappa}}\AA_{n+1}\quad\text{and}\quad\AA^{\UUU}\cong\bigcup_n\AA_n.
    \]
    and similarly for the models $\BB_n$.
    The chains $(\AA_n)_n$ and $(\BB_n)_n$ are thus $\LL_{\kappa\kappa}$-elementary, and by the Union Lemma \ref{union},
    \[
    \AA\equiv_{\LL_\kappa^1}\AA^{\UUU}\cong\BB^{\UUU}\equiv_{\LL^1_\kappa}\BB.
    \]
    
    This shows that, indeed, $\AA\equiv_{\LL_\kappa^1}\BB$, as wanted.
\end{proof}

\nocite{Barwise1985-BARML-8}
\nocite{kanamori2008higher}
\bibliographystyle{plain}
\bibliography{bib}

\end{document}